\documentclass[11pt]{article}
\usepackage{amsthm}
\usepackage{amsmath}
\usepackage{amssymb}
\usepackage{latexsym}
\usepackage{amsfonts}
\usepackage{placeins}
\usepackage{caption}
\usepackage{mathrsfs}
\usepackage[latin1]{inputenc}
\usepackage{listings}
\newtheorem{theorem}{\bf Theorem}[section]

\newtheorem{lemma}[theorem]{\bf Lemma}
\begin{document}
	\title{On balancing and Lucas-balancing numbers expressible as product
		of two $k$-Fibonacci numbers}
	\author{Bibhu Prasad Tripathy and Bijan Kumar Patel}
	\date{}
	\maketitle
	\begin{abstract}
		A positive integer $n$ is called a balancing number if there exists a positive integer $r$ such that $1 + 2 + \cdots + (n-1) = (n+1) + (n+2) + \cdots + (n+r)$. The corresponding value $r$ is known as the balancer of $n$. If $n$ is a balancing number, then $8n^{2}+1$ is a perfect square, and its positive square root is called a Lucas-balancing number. For any integer $k \geq 2$, let $\{F_{n}^{(k)} \}_{n \geq -(k-2)}$ denote $k$-generalized Fibonacci sequence which starts with $0, \dots ,1$($k$ terms) where each next term is the sum of the $k$ preceding terms. In this paper, we investigate all balancing and Lucas-balancing numbers that can be expressed as the product of two $k$-generalized Fibonacci numbers. 
	\end{abstract} 
	
	\noindent \textbf{\small{\bf Keywords}}: $k$-Fibonacci numbers, balancing numbers, Lucas-balancing numbers, linear forms in logarithms, reduction method. \\
	{\bf 2020 Mathematics Subject Classification:} 11B39; 11J86.
	
\section{Introduction}
A positive integer $n$ is said to be a balancing number if
	\[
	\sum_{i=1}^{n-1} i = \sum_{j= n +1}^{n+r} j
	\]
holds for some positive integer $r$. Then $r$ is referred to as the balancer corresponding to the balancing number $n$ (see \cite{Behera}). The sequence of balancing numbers, denoted by $\{B_{n}\}_{n \geq 0}$, is defined by the recurrence relation  
	\[
B_{0} = 0, ~B_{1} = 1,~	B_{n+1} = 6B_{n} - B_{n-1}, \quad \text{for all} \quad n \geq 1.
	\]
Associated with each balancing number $n$, the positive value of $\sqrt{8 n^{2}+1}$ is known as the Lucas-balancing number, denoted by $C_{n}$ (see \cite{Panda, Ray}). The sequence $\{C_{n}\}_{n \geq 0}$, called the Lucas-balancing sequence, satisfies the same recurrence relation as $\{ B_{n} \}$, but with different initial values 
\[
C_{0} = 1, ~C_{1} = 3,~	C_{n+1} = 6C_{n} - C_{n-1}, \quad \text{for all} \quad n \geq 1.
	\]
The balancing and Lucas-balancing sequences correspond to entries A001109 and A001541, respectively, in the On-Line Encyclopedia of Integer Sequences (OEIS). Both sequences satisfy the characteristic equation $x^{2} - 6x + 1 = 0$, whose roots are $\gamma = 3+ 2 \sqrt{2}$ and $\delta =  3 - 2 \sqrt{2}$. The Binet formulas for $B_{n}$ and $C_{n}$ are given by
	\begin{equation}\label{eq 1.1}
		B_{n} = \frac{\gamma^{n} - \delta^{n}}{4 \sqrt{2}} \quad \text{and} \quad C_{n} = \frac{\gamma^{n} + \delta^{n}}{2} \quad \text{for all} \quad n \geq 1.   \end{equation}
By induction, it can be shown that the inequalities
	\begin{equation}\label{eq 1.2}
		\gamma^{n-1} \leq B_{n} \leq \gamma^{n} \quad \quad \text{and} \quad \quad \gamma^{n} \leq 2 C_{n} \leq \gamma^{n+1} \quad \text{hold for all} \quad n \geq 1.
	\end{equation}

The classical Fibonacci sequence $\{F_n\}_{n\geq 0}$ is the binary recurrence sequence defined by 
	\[
	F_{n+2} = F_{n+1} + F_n \quad \text{for all} \quad n\geq 0,
	\]
with initials $F_0 = 0$ and $F_1 = 1$. Let $ k \geq 2 $ be an integer. We consider a generalization of the Fibonacci sequence $\{F_{n}^{(k)} \}_{n \geq -(k-2)}$ defined as 
	\begin{equation}\label{eq 1.3}
		F_{n}^{(k)} = F_{n-1}^{(k)} + F_{n-2}^{(k)} + \dots + F_{n-k}^{(k)} = \sum_{i=1}^{k} F_{n-i}^{(k)} ~~\text{for all}~n \geq 2, 
	\end{equation}
with initials $F_{-(k-2)}^{(k)} = F_{-(k-3)}^{(k)} = \dots = F_{0}^{(k)} = 0$ and $F_{1}^{(k)} = 1$. This sequence $F_{n}^{(k)}$ is called the $k$-generalized Fibonacci sequence or the $k$-Fibonacci sequence. This generalization is a family of sequences, with each new choice of $k$ producing a unique sequence. For example, if $k = 2$, we get $F_{n}^{(2)} = F_{n}$, the classical Fibonacci sequence. For $k = 3$, we have $F_{n}^{(3)} = T_n$, the $n$th Tribonacci number. They are followed by the Tetranacci numbers for $k = 4$, and so on.
	
Diophantine equations have a long history involving the intersection of $k$-Fibonacci and other linear recurrence sequences. In recent years, the study of finding intersections of two linear recurrence sequences has been extended to the broader problem of determining all terms in a given sequence that can be expressed as products of terms from another sequence. For instance, Ddamulira et al. \cite{Ddamulira} investigated Fibonacci numbers that are products of two Pell numbers, as well as Pell numbers that are products of two Fibonacci numbers. In a related study, Ad\'edji et al. \cite{Adedji} examined Padovan and Perrin numbers that are products of two generalized Lucas numbers. Liptai \cite{Liptai} has shown that $1$ is the only balancing number appearing in the Fibonacci sequence. Rihane \cite{Rihane3} generalized this result and proved that $1, 6930$ are the only balancing numbers and $1, 3$ are the only Lucas-balancing numbers, that occur as terms of a  $k$-generalized Fibonacci sequence. Consequently, Erduvan and Keskin \cite{Erduvan} determined all Fibonacci numbers that are products of two balancing numbers and, conversely, all balancing numbers that are products of two Fibonacci numbers. Recently, Rihane \cite{Rihane4} extended the work of Erduvan and Keskin \cite{Erduvan} by determining all $k$-Fibonacci numbers that are the product of two balancing or Lucas-balancing numbers.  

Motivated by the above works, the present study extends the work of Erduvan and Keskin \cite{Erduvan} and aims to determine the balancing and Lucas-balancing numbers that are products of two $k$-Fibonacci numbers. More precisely, we establish the following results.

\begin{theorem}\label{thm1}
		All the solutions of the Diophantine equation 
		\begin{equation}\label{eq 1.4}
			B_{l} = F_{n}^{(k)} F_{m}^{(k)},
		\end{equation}
		in positive integers $n, m, l$ and $k$ with $k \geq 3$ and $1 \leq m \leq n$ are
		\[
		B_{1} = F_{1}^{(k)} F_{1}^{(k)} \quad \text{for all} \quad k \geq 3, \quad B_{1} = F_{1}^{(k)} F_{2}^{(k)} \quad \text{for all} \quad k \geq 3, \quad B_{1} = F_{2}^{(k)} F_{2}^{(k)} \quad \text{for all} \quad k \geq 3,
		\]
		\[
		 B_{6} = F_{1}^{(5)} F_{15}^{(5)} \quad \text{and} \quad B_{6} = F_{2}^{(5)} F_{15}^{(5)}.
		\]
\end{theorem}

\begin{theorem}\label{thm2}
All the solutions of the Diophantine equation
\begin{equation}\label{eq 1.5}
	C_{l} = F_{n}^{(k)} F_{m}^{(k)},
		\end{equation}
		in positive integers $n, m, l$ and $k$ with $k \geq 2$ and $1 \leq m \leq n$ are $C_{1} = F_{1}^{(2)} F_{4}^{(2)} \ \text{and} \ C_{1} = F_{2}^{(2)} F_{4}^{(2)}$.	
\end{theorem}

The proofs of our theorems rely primarily on linear forms in logarithms of algebraic numbers, as developed by Matveev \cite{Matveev}, and on the reduction algorithm of Dujella and Peth\"{o} \cite{Dujella}. We begin by presenting a few preliminary results, which are discussed in detail in the subsequent section.	

\section{Preliminary Results}
	This section is devoted to gathering several definitions, notations, properties, and results that will be used in the rest of this study.
	\subsection{Linear forms in logarithms}
	Let $\gamma$ be an algebraic number of degree $d$ with a minimal primitive polynomial 
	\[
	f(Y):= b_0 Y^d+b_1 Y^{d-1}+ \cdots +b_d = b_0 \prod_{j=1}^{d}(Y- \gamma^{(j)}) \in \mathbb{Z}[Y],
	\]
	where the $b_j$'s are relatively prime integers, $b_0 >0$, and the $\gamma^{(j)}$'s are conjugates of $\gamma$. Then the \emph{logarithmic height} of $\gamma$ is given by
	\begin{equation*}
		h(\gamma)=\frac{1}{d}\left(\log b_0+\sum_{j=1}^{d}\log\left(\max\{|\gamma^{(j)}|,1\}\right)\right).
	\end{equation*}
	With the above notation, Matveev (see  \cite{Matveev} or  \cite[Theorem~9.4]{Bugeaud}) proved the following result.
	
	\begin{theorem}\label{thm3}
		Let $\eta_1, \ldots, \eta_s$ be positive real algebraic numbers in a real algebraic number field $\mathbb{L}$ of degree $d_{\mathbb{L}}$. Let $a_1, \ldots, a_s$ be non-zero  integers such that
		\[
		\Lambda :=\eta_1^{a_1}\cdots\eta_s^{a_s}-1 \neq 0.
		\]
		Then
		\[
		- \log  |\Lambda| \leq 1.4\cdot 30^{s+3}\cdot s^{4.5}\cdot d_{\mathbb{L}}^2(1+\log d_{\mathbb{L}})(1+\log D)\cdot B_1 \cdots B_s,
		\]
		where
		\[
		D\geq \max\{|a_1|,\ldots,|a_s|\},
		\]
		and
		\[
		B_j\geq \max\{d_{\mathbb{L}}h(\eta_j),|\log \eta_j|, 0.16\}, ~ \text{for all} ~ j=1,\ldots,s.
		\]
	\end{theorem}

	\subsection{The reduction method}
	Our next tool is a version of the reduction method of Baker and Davenport (see \cite{Baker}). Here, we use a slight variant of the version given by Dujella and Peth\"{o} (see \cite{Dujella}). For a real number $x$, we write $||x||$ for the distance from  $x$ to the nearest integer.
	
	\begin{lemma}\label{lem 2.2}
		Let $M$ be a positive integer, $p/q$ be a convergent of the continued fraction of the irrational $\tau$ such that $q > 6M$, and $A, B, \mu $ be some real numbers with $A>0$ and $B>1$. Furthermore, let
		\[\epsilon:=||\mu q|| - M \cdot ||\tau q||.
		\]
		If $\epsilon >0$, then there is no solution to the inequality 
		\[
		0< |u \tau - v + \mu| <AB^{-w}
		\]
		in positive integers $u$, $v$ and $w$ with
		\[
		u \leq M \quad\text{and}\quad w \geq \frac{\log(Aq/\epsilon)}{\log B}.
		\]	
	\end{lemma}
	The above lemma cannot be applied when $\mu$ is a linear combination of 1 and $\tau$, since then $\epsilon < 0$. In this case, we use the following nice property of continued fractions (see Theorem 8.2.4 and top of page 263 in \cite{Murty})
	
	\begin{lemma}\label{lem 2.3}
		Let $\tau$ be an irrational number, $\frac{p_0}{q_0}, \frac{p_1}{q_1}, \frac{p_2}{q_2}, \dots$ be all the convergents of the continued fraction of $\tau$, and $M$ be a positive integer. Let $N$ be a non-negative integer such that $q_N > M$. Then putting $a(M) := \max \{a_i: i=0,1,2,\dots, N \}$, the inequality
		\[
		\Bigm| x \tau - y \Bigm| > \frac{1}{(a(M)+2)x},
		\]
		holds for all $x < M$.
	\end{lemma}
	
	\subsection{Properties of $k$-Fibonacci sequence}
	We recall some of the facts and properties of the $k$-generalized Fibonacci sequence which will be used after. Note that the characteristic polynomial of the $k$-generalized Fibonacci sequence is
	\[
	\Psi_{k}(x) = x^k - x^{k-1} - \dots - x - 1.
	\]
	$\Psi_{k}(x)$ is irreducible over $\mathbb{Q}[x]$ and has just one root outside the unit circle. It is real and positive, so it satisfies $\varphi(k) > 1$. The other roots are strictly inside the unit circle. Throughout this paper, $\varphi := \varphi(k)$ denotes that single root, which is located between $2(1-2^{-k})$ and $2$ (see \cite{Wolfram}). To simplify notation, we will omit the dependence on $k$ of $\varphi$. Dresden and Du \cite{Dresden} gave a simplified Binet-like formula for $F_{n}^{(k)}$:
	\begin{equation}\label{eq 2.7}
		F_{n}^{(k)} = \displaystyle\sum_{i=1}^{k} f_{k} (\varphi_{i})\varphi_{i}^{n-1} = \sum_{i=1}^{k} \frac{\varphi_{i}-1}{2+(k+1)(\varphi_{i}-2)} \varphi_{i}^{n-1} ,
	\end{equation}
	where  $\varphi := \varphi_{1}, \varphi_{2},\dots , \varphi_{k}$  are the roots of the characteristic polynomial $\Psi_{k}(x)$. It was also shown in \cite{Dresden} that the contribution of the roots lying inside the unit circle to the formula \eqref{eq 2.7} is very small, which is given by the approximation 
	\begin{equation}\label{eq 2.8}
		\left| F_{n}^{(k)} -  f_{k} (\varphi) \varphi^{n-1} \right| < \frac{1}{2}
	\end{equation}
		holds for all $n \geq 1$ and $k \geq 2$. So, for $n \geq 1$ and $k \geq 2$, we have
	\begin{equation}\label{eq 2.9}
		F_{n}^{(k)} = f_{k}(\varphi) \varphi^{n-1} + e_{k}(n), \quad \text{where} \quad |e_{k}(n)| \leq \frac{1}{2}. 
	\end{equation}
	Furthermore, it was shown by Bravo and Luca in \cite{ Bravo} that the inequality
	\begin{equation}\label{eq 2.10}
		\varphi ^ {n-2} \leq F_{n} ^ {(k)} \leq \varphi ^ {n-1} \text{ holds for all } n \geq 1 ~{\rm and}~ k \geq 2.
	\end{equation}
	One can observe that the first $k + 1$ non-zero terms in $F_{n}^{(k)}$ are powers of $2$, namely
	\[
	F_{1}^{(k)}=1, F_{2}^{(k)}=1, F_{3}^{(k)}=2, F_{4}^{(k)}=4, \dots, F_{k+1}^{(k)} = 2^{k-1},
	\]
	while the next term in the above sequence is $F_{k+2}^{(k)} = 2^k -1$. Thus, we have that
	\[
	F_{n}^{(k)} = 2^{n-2}~\text{holds for all}~2 \leq n \leq k+1.
	\]
	The following result was proved by Bravo and Luca \cite{Bravo}.
	\begin{lemma}\label{lem 2.4}
		Let $k \geq 2$, $\varphi$ be the dominant root of   $\{F_{n}^{(k)}\}_{n \geq -(k-2)}$, and consider the function defined in \eqref{eq 2.7}. Then 
		\[
		\frac{1}{2} < f_{k}(\varphi) < \frac{3}{4} \quad \text{and} \quad \left|f_{k}(\varphi^{(i)})  \right| < 1, \quad  2 \leq i \leq k.
		\]	
	\end{lemma}
	\noindent
	In addition they proved that the logarithmic height of $f$ satisfies 
	\begin{equation}\label{eq 2.11}
		h(f_{k}(\varphi)) < \log(k+1) + \log4 \quad \text{for all} \quad k \geq 2.
	\end{equation}
	Bravo et al. \cite{Bravo1} proved that for all $n \geq k+2$, we have
	\begin{equation}\label{eq 2.12}
		F_{n}^{(k)} = 2^{n-2}(1 + \xi), \quad \text{where} \quad  |\xi| < \frac{1}{2^{k/2}}.
	\end{equation}

\subsection{Other useful lemmas}
We conclude this section by recalling two lemmas that we will need in this work.
	
	\begin{lemma}\label{lem 2.5} {\rm{(\cite{Weger}, Lemma 2.2)}}
		Let $a, x \in \mathbb{R}$. If $0< a < 1$ and $|x| < a$, then 
		\[
		|\log(1+x)| < \frac{-\log(1-a)}{a}\cdot |x|
		\]
		and 
		\[
		|x| < \frac{a}{1 - e^{-a}} \cdot |e^{x} - 1|.
		\]
	\end{lemma}
	
	\begin{lemma}\label{lem 2.6} {\rm{(\cite{Sanchez}, Lemma 7)}}
		If $m \geq 1$, $S \geq (4m^{2})^{m}$ and $\frac{x}{(\log x)^{m}} < S$, then $x < 2^{m} S (\log S)^{m}$.
	\end{lemma}

\section{Proof of Theorem \ref{thm1}}
We begin our analysis of \eqref{eq 1.4} for $2 \leq n \leq k+1$. In this case, it is known that $F_{n}^{(k)} = 2^{n-2}$, thus the equation \eqref{eq 1.4} transform into
	\[
	B_{l} = 2^{m+n-4},
	\]
which has no solution in the range $2 \leq n \leq k+1$. From now on, we consider that $n \geq k+2$ and $k \geq 3$.

\subsection{An upper bound for l versus n}
Combining inequalities \eqref{eq 1.2} and \eqref{eq 2.10} with the equation \eqref{eq 1.4}, we have
	\[
	\gamma^{l-1} \leq B_{l} =  F_{n}^{(k)} F_{m}^{(k)} \leq  \varphi^{n+m-2} < \varphi^{2n-2}.
	\]
This leads us to the inequality
	\[
	l \leq 1 + (2n - 2) \left( \frac{\log \varphi}{\log \gamma}  \right).
	\]
By using this and taking into account that $2(1-2^{-k}) < \varphi(k) < 2$ for all $k \geq 3$, it result
	\begin{equation}\label{eq 3.13}
		l < 0.8 n + 0.2.    
	\end{equation}
	
\subsection{Bounding $n$ in terms of $k$}
In this subsection, we will bound $n$ in terms of $k$. Namely, we will show the following lemma.
\begin{lemma}\label{lem 3.1}
If $(l, k, n, m)$ is an integer solution   of \eqref{eq 1.4} with $k \geq 3$ and $n \geq k+2$, then the inequalities 
\begin{equation}\label{eq 3.14}
n < 5.1 \times 10^{31} k^{8} \log^{5} k
\end{equation}
hold.
\end{lemma}
\begin{proof}
Using Binet's formulas \eqref{eq 1.1} and \eqref{eq 2.9}, we rewrite the equation \eqref{eq 1.4} as
\begin{equation}\label{eq 3.15}
		\frac{\gamma^{l} - \delta^{l}}{4 \sqrt{2}} = 
			\left(f_{k}(\varphi) \varphi^{n-1} +e_{k}(n) \right) \left(f_{k}(\varphi) \varphi^{m-1} +e_{k}(m) \right),
\end{equation}
which implies
	\[
		\frac{\gamma^{l}}{4 \sqrt{2}} - f_{k}^{2}(\varphi) \varphi^{n+m-2}  = f_{k}(\varphi) \varphi^{n-1} e_{k}(m) + f_{k}(\varphi) \varphi^{m-1} e_{k}(n) + e_{k}(n) e_{k}(m) + \frac{\delta^{l}}{4 \sqrt{2}}. 
	\]
Thus, we obtain
\begin{equation}\label{eq 3.16}
	\left| \frac{\gamma^{l}}{4 \sqrt{2}} - f_{k}^{2}(\varphi) \varphi^{n+m-2} \right| < \frac{ f_{k}(\varphi)}{2} \varphi^{n-1} + \frac{ f_{k}(\varphi)}{2} \varphi^{m-1} + \frac{1}{4} + \frac{|\delta|^{l}}{4 \sqrt{2}}.
\end{equation}
Dividing the above inequality by $f_{k}^{2}(\varphi) \varphi^{n+m-2}$ and using the fact $f_{k}(\varphi) > 1/2$, we get
\begin{equation}\label{eq 3.17}
		\left| \gamma^{l} \varphi^{-(n+m-2)} \frac{ f_{k}^{-2}(\varphi)}{4 \sqrt{2}} -1 \right| < \frac{1}{\varphi^{m-1}} + \frac{1}{\varphi^{n-1}} + \frac{4}{\varphi^{n+m-2}} < \frac{6}{\varphi^{m-1}}.
\end{equation}
Let
\begin{equation}\label{eq 3.18}
		\Lambda_{1} :=  \gamma^{l} \varphi^{-(n+m-2)} \frac{f_{k}^{-2}(\varphi)}{4 \sqrt{2}} -1.
\end{equation}
From \eqref{eq 3.17}, we have 
\begin{equation} \label{eq 3.19}
	|\Lambda_{1}| < 6 \cdot \varphi^{-(m-1)}.
\end{equation}
Suppose that $\Lambda_{1}= 0$, then we get
	\[
		\frac{\gamma^{l}} {4 \sqrt{2}} = f_{k}^{2}(\varphi) \varphi^{n+m-2}.
	\]        
Conjugating the above relation with an automorphism $\sigma$ of the Galois extension of $\mathbb{Q}(\varphi, \gamma)$ over $\mathbb{Q}$ given by $\sigma(\varphi)=\varphi_{i} \ (i > 1)$ and $\gamma \to \gamma$, we obtain
\[
 \frac{\gamma^{l}} {4 \sqrt{2}} = f_{k}^{2}(\varphi_{i}) \varphi_{i}^{n+m-2} 
\]
for some $i > 1$, where  $\varphi_{i}$ represents the roots of the characteristic polynomial $\Psi_{k}(x)$.  Using Lemma \ref{lem 2.4} with the fact $|\varphi_{i}| < 1$ and taking absolute value in the above relation we see that this is not possible since its right-hand side exceeds $1$ for all $l \geq 1$, while its left-hand side is smaller than 1. Hence $\Lambda_{1} \neq 0$. Therefore, we apply Theorem \ref{thm3} to get a lower bound for $\Lambda_{1}$ given by \eqref{eq 3.18} with the parameters:
\[
	\eta_{1} := \gamma, \quad \eta_{2} := \varphi , \quad \eta_{3} := \frac{f_{k}^{-2}(\varphi)}{4 \sqrt{2}},
\]
and
\[ 
a_{1}:= l, \quad a_{2}:= -(n+m-2), \quad a_{3}:= 1.
\]
Note that the algebraic numbers $\eta_{1}, \eta_{2}, \eta_{3}$ belongs to the field  $\mathbb{L} := \mathbb{Q}(\varphi, \sqrt{2})$, so we can assume $d_{\mathbb{L}} = [\mathbb{L}:\mathbb{Q}] = 2k$. Since $h(\eta_{1}) = (\log \gamma) / 2$ and $h(\eta_{2}) = (\log \varphi) / k < (\log 2) /k$, it follows that 
\[
	\max\{2kh(\eta_{1}),|\log \eta_{1}|,0.16\} = k \log \gamma := B_{1}
\]
and 
\[
	\max\{2kh(\eta_{2}),|\log \eta_{2}|,0.16\} = 2 \log 2 := B_{2}.
\]
Using the estimate \eqref{eq 2.11} and the properties of logarithmic height, it follows that for all $k \geq 3$
\begin{align*}
		h(\eta_{3})  & < 2 h(f_{k}(\varphi)) + h(4 \sqrt{2}) \\
			& < 2 \left( \log(k+1) + \log 4 \right) +  \frac{\log 32}{2} \\
			& < 6.7 \log k.  
\end{align*}
Thus, we obtain
		\[
		\max\{2kh(\eta_{3}),|\log \eta_{3}|,0.16\} = 13.4 k \log k := B_{3}.
		\]
Finally, the fact that $m \leq n$ and the inequality \eqref{eq 3.13} imply that we can take $D := 2 n$. Therefore, according to Theorem \ref{thm3}, it comes that 
		\begin{equation}\label{eq 3.20}
			\log |\Lambda_{1}| >   -1.432 \times 10^{11} (2k)^{2} (1+ \log 2k) (1 +\log 2n) (k \log \gamma) (2 \log 2) (13.4 k \log k).   
		\end{equation}
From the comparison of lower bound \eqref{eq 3.20} and upper bound \eqref{eq 3.19} of   $|\Lambda_{1}|$ gives us
\[
	(m-1) \log \varphi - \log 6 < 1.88 \times 10^{13} k^{4} \log k (1 + \log 2k)(1+ \log 2n).
\]
Using the facts $1 + \log 2k < 2.6 \log k$ for all $k \geq 3$ and $1+ \log 2n < 2.1 \log n$ for all $n \geq 5$, we conclude that 
\begin{equation}\label{eq 3.21}
		(m-1)  \log \varphi < 1.2 \times 10^{14} k^{4} \log^{2} k \log n.    
\end{equation}
In order to apply Theorem \ref{thm3} a second time, we go back to \eqref{eq 1.4} and we express it as
\[
	\frac{\gamma^{l} - \delta^{l}}{4 \sqrt{2}} = 
		\left(f_{k}(\varphi) \varphi^{n-1} +e_{k}(n) \right) F_{m}^{(k)}.
\]
From the above, it follows
\begin{equation}\label{eq 3.22}
		\left| \frac{\gamma^{l}}{4 \sqrt{2}F_{m}^{(k)}} - f_{k}(\varphi) \varphi^{n-1} \right|= \left| \frac{\delta^{l}}{4 \sqrt{2}F_{m}^{(k)}} + e_{k}(n) \right| < 1.5.
\end{equation}
If we divide through by $f_{k}(\varphi) \varphi^{n-1}$, we obtain 
		\begin{equation}\label{eq 3.23}
			|\Lambda_{2}| \leq  \frac{1.5}{f_{k}(\varphi) \varphi^{n-1}}  < \frac{3}{\varphi^{n-1}}, \end{equation}
		where
		\begin{equation}\label{eq 3.24}
			\Lambda_{2} := \gamma^{l}  \varphi^{-(n-1)} \frac{ f^{-1}_{k}(\varphi)}{4 \sqrt{2}F_{m}^{(k)} } - 1.
		\end{equation} 
We can prove that $\Lambda_{2} \neq 0$ by a similar method used to show that $\Lambda_{1} \neq 0$.        
Now, let us apply Theorem \ref{thm3} with
		\[
		(\eta_{1}, a_{1}) := \left(\gamma, l \right), \quad \quad (\eta_{2}, a_{2}) := \left( \varphi, -(n-1) \right),~ \quad {\rm and} ~  \quad (\eta_{3}, a_{3}) := \left(\frac{f^{-1}_{k}(\varphi)}{4 \sqrt{2} F_{m}^{(k)} }, 1 \right).
		\]
		The number field containing $\eta_{1}, \eta_{2}, \eta_{3}$ is  $\mathbb{L} := \mathbb{Q}(\varphi, \sqrt{2})$, which has degree $d_{\mathbb{L}} = [\mathbb{L}:\mathbb{Q}]=2k$. As calculated before, we take 
		\[
		B_{1} = k \log \gamma, \quad \quad B_{2} =  2 \log 2.
		\]
		But we must calculate $h(\eta_{3})$ and then $B_{3}$. Using logarithmic properties and \eqref{eq 3.21}, we deduce that for all $k \geq 3$,
		\begin{align*}
			h(\eta_{3})  & \leq h\left( \frac{ f^{-1}_{k}(\varphi)}{4 \sqrt{2} F_{m}^{(k)}} \right) \\
			& \leq  h(f_{k}(\varphi)) + h(4 \sqrt{2}) + h(F_{m}^{(k)}) \\
			& < \log(k+1) + \log4  + \frac{\log 32}{2} + \log 2 + (m-1) \log \varphi \\
			& < 1.3 \times 10^{14} k^{4} \log^{2} k \log n.  
		\end{align*}
		Thus, we conclude that
		\[
		\max\{2kh(\eta_{3}),|\log \eta_{3}|,0.16\} = 2.6 \times 10^{14} k^{5} \log^{2} k \log n := B_{3}.
		\]
		Finally, by the inequalities \eqref{eq 3.13} and $l < 0.8 n + 0.2 < n$ which hold for all $n \geq 5$, it seen that we can take $D: = n$. Therefore, applying Theorem \ref{thm3} and comparing the resulting inequality with \eqref{eq 3.23}, we obtain
		\[
		n < 2.91 \times 10^{27} k^{8} \log^{3} k \log^{2} n,
		\]
		where we have used the facts $1 + \log 2k < 2.6 \log k$ and $1+ \log n < 1.7 \log n$ for all $k \geq 3$ and $n \geq 5$. Therefore, we obtain
		\begin{equation}\label{eq 3.25}
			\frac{n}{\log^{2} n} < 2.91 \times 10^{27} k^{8} \log^{3} k.
		\end{equation}
		Thus, putting  $S := 2.91 \times 10^{27} k^{8} \log^{3} k$ in \eqref{eq 3.25} and using Lemma \ref{lem 2.6} with the fact $63.23 + 8 \log k + 3 \log (\log k) < 66 \log k$ for all $k \geq 3$, gives 
		\begin{align*}
			n & < 2^2 \left(2.91 \times 10^{27} k^{8} \log^{3}k \right) (\log \left(2.91 \times 10^{27} k^{8} \log^{3}k \right))^2 \\
			& < 2^2 \left(2.91 \times 10^{27} k^{8} \log^{3} k \right) (63.23 + 8 \log k + 3 \log (\log k))^2 \\
			&< 5.1 \times 10^{31} k^{8} \log^{5} k.
		\end{align*}
		This establishes \eqref{eq 3.14} and finishes the proof of Lemma \ref{lem 3.1}.
	\end{proof}

	\subsection{The case when $3 \leq k \leq 450$}
	In this subsection, we will show the following result. 
	\begin{lemma}\label{lem 3.2}
		If $(l, k, n, m)$ is an integer solution   of \eqref{eq 1.4} with $3 \leq k \leq 450$ and $n \geq k+2$, then our variables are bounded as follows:
		\[
		n  \leq 409 \quad \text{and} \quad l \leq 327.
		\]
	\end{lemma}
	\begin{proof}
		To apply Lemma \ref{lem 2.2}, we define
		\begin{equation}\label{eq 3.26}
			\Gamma_{1} := \log (\Lambda_{1} + 1) = l \log \gamma -(n+m-2) \log \varphi + \log \left( \frac{f_{k}^{-2}(\varphi)}{4 \sqrt{2}} \right).   
		\end{equation}
		Suppose that $m \geq 5$, then by \eqref{eq 3.19}, we have $|\Lambda_{1}| < 0.64$. Choosing $a := 0.64$, we obtain the inequality
		\[
		|\Gamma_{1}| < \frac{- \log(1 - 0.64)}{0.64} \cdot \frac{6}{\varphi^{m-1}} < \frac{9.6}{\varphi^{m-1}}
		\]
		by Lemma \ref{lem 2.5}. Thus, it follows that
		\begin{equation*}
			\left|l \log \gamma -(n+m-2) \log \varphi + \log \left( \frac{f_{k}^{-2}(\varphi)}{4 \sqrt{2}} \right) \right| <  \frac{9.6}{\varphi^{m-1}}. 
		\end{equation*}
		Dividing the above inequality by $\log \varphi$, we get 
		\begin{equation}\label{eq 3.27}
			\left| l  \left( \frac{\log \gamma}{\log \varphi} \right) - (n+m) + 2 +  \frac{\log \left(\frac{ f_{k}^{-2}(\varphi)}{4 \sqrt{2}}\right)}{\log \varphi} \right| < \frac{17.2}{\varphi^{m-1}}.   
		\end{equation}
		In order to apply Lemma \ref{lem 2.2}, we set
		\[ \tau =  \frac{\log \gamma}{\log \varphi}, \quad \mu = 2 +  \frac{\log \left( \frac{f_{k}^{-2}(\varphi)}{4 \sqrt{2}}\right)}{\log \varphi}, \quad A:= 17.2 \quad \text{and} \quad B:= \varphi.
		\]
		We have $\tau \notin \mathbb{Q}$. Indeed, if we assume there exist coprime integers $a$ and $b$ such that $\tau = a/b$, then we get that $\varphi^{a} = \gamma^{b}$. Let $\sigma \in Gal(\mathbb{K} / \mathbb{Q})$, the Galois group of the extension $\mathbb{K} / \mathbb{Q}$, such that $\sigma(\varphi)=\varphi_{i}$, for some $i \in \{3, \ldots, k\}$. Applying this to the above relation and taking absolute values we get $1 < 2^{a} = |\varphi_{i}| < 1$, which is a contradiction. Moreover, we note that $ M_{k}:= \lfloor 5.1 \times 10^{31} k^{8} \log^{5} k   \rfloor$, which is an upper bound of $l$ from Lemma \ref{lem 3.1}. For each $k \in [3, 450]$, we find a good approximation of $\tau$ and a convergent $p_{l}/q_{l}$ of the continued fraction of $\tau$ such that $ q = q(k) > 6M_{k}$ is a denominator of a convergent of the continued fraction of $\tau$ with $ \epsilon = \epsilon(k):= \| \mu q\| - M_{k}\|\tau q\| > 0$. A computer search with \textit{Mathematica} revealed that if $k \in [3,450]$, then the maximum value of $\log(Aq/\epsilon)/ \log B$ $442.771$, which is according to Lemma \ref{lem 2.2}, is an upper bound on $m-1$.  
		
		Now, we fix $1 \leq m \leq 443$ and we consider
		\begin{equation}\label{eq 3.28}
			\Gamma_{2} := \log (\Lambda_{2} + 1) = l \log \gamma -(n-1) \log \varphi + \log \left( \frac{ f_{k}^{-1}(\varphi)}{4 \sqrt{2} F_{m}^{(k)}}  \right).   
		\end{equation}
		It is easy to see that $|\Lambda_{2}| < 0.98$ for $n \geq 3$. As a result, using inequality \eqref{eq 3.23} and Lemma \ref{lem 2.5}, we obtain 
		\begin{equation}\label{eq 3.29}
			|\Gamma_{2}| < \frac{- \log(1 - 0.98)}{0.98} \cdot \frac{3}{\varphi^{n-1}} < \frac{12}{\varphi^{n-1}}.    
		\end{equation}
		Replacing \eqref{eq 3.29} by \eqref{eq 3.28} and dividing through by $\log \varphi$, we obtain
		\begin{equation}\label{eq 3.30}
			\left| l \left( \frac{\log \gamma}{\log \varphi} \right) - (n-1) +   \frac{\log \left( \frac{ f_{k}^{-1}(\varphi)}{4 \sqrt{2} F_{m}^{(k)}}\right)}{\log \varphi} \right| <  \frac{21.45}{\varphi^{n-1}}. 
		\end{equation}
		To apply Lemma \ref{lem 2.2} to \eqref{eq 3.30}, this time for $1 \leq m \leq 443$, we set
		\[ \tau =  \frac{\log \gamma}{\log \varphi}, \quad \mu = 1 + \frac{\log \left( \frac{f_{k}^{-1}(\varphi)}{4 \sqrt{2} F_{m}^{(k)}}\right)}{\log \varphi}, \quad A:= 21.45 \quad \text{and} \quad B:= \varphi.
		\]
		Again, for $(k, m) \in [3, 450] \times [1, 443]$, we find a good
		approximation of $\tau$ and a convergent $p_{l}/q_{l}$ of the continued fraction of $\tau$ such that $ \epsilon = \epsilon(k):= \| \mu q\| - M_{k}\|\tau q\| > 0$, where $M_{k}:= \lfloor 5.1 \times 10^{31} k^{8} \log^{5} k   \rfloor$ which is an upper bound of $l$ from Lemma \ref{lem 3.1}. After doing this, we use Lemma \ref{lem 2.2} on inequality \eqref{eq 3.30}. Again, a program in \textit{Mathematica} revealed that the maximum value of $\log(Aq/\epsilon)/ \log B$ over all $(k, m) \in [3, 450] \times [1, 443]$ is $408.668$, which according to Lemma \ref{lem 2.2}, is an upper bound of $n-1$.
		
		Hence, we deduce that the possible solutions $(l, k, n, m)$ of \eqref{eq 1.4} for which
		$k \in [3, 450]$ satisfy $k+2 \leq m \leq n \leq 409$. Therefore, we use inequality \eqref{eq 3.13} to obtain $l \leq 327$.
	\end{proof}
	
	\subsection{The case when $k > 450$}
	In this subsection, our primary goal is to show the following lemmas that aim to prove that there does not exist any solution for the case when $k > 450$ and $n \geq k+2$.
	
	\begin{lemma}\label{lem 3.3}
		If $(l, k, m, n)$ is a solution of the Diophantine equation \eqref{eq 1.4} with $k > 450$ and $n \geq k+2$, then $k$ and $n$ are bounded as
		\[
		k < 5.7 \times 10^{16} \quad \text{and} \quad n < 4.86 \times 10^{173}.
		\]
	\end{lemma}
	\begin{proof}
		For $k > 450$, the following inequalities hold
		\[
		n < 5.1 \times 10^{31} k^{8} \log^{5} k < 2^{k/2}.
		\]
		Hence, from \eqref{eq 2.12}, we have
		\begin{equation}\label{eq 3.31}
			F_{n}^{(k)}F_{m}^{(k)} = 2^{m+n-4} (1 + \xi)^2
		\end{equation}
		Inserting \eqref{eq 3.31} in \eqref{eq 1.4}, we obtain
		\begin{align*}
			\frac{\gamma^{l}}{4 \sqrt{2}} -  2^{m+n-4} = 2^{m+n-4} \left( 2 \xi + \xi^{2} \right) + \frac{\delta^{l}}{4 \sqrt{2}}.
		\end{align*}
		The above inequality and the fact $\left( 2 \xi + \xi^{2} \right) < 3/2^{k/2}$ together with $n \geq m \geq k+2$ yeilds that
		\begin{equation}\label{eq 3.32}
			\left| \gamma^{l} 2^{-(m+n-2)} (\sqrt{2})^{-1} - 1  \right| < \frac{3}{2^{k/2}} + \frac{1}{2^{2k}} < \frac{4}{2^{k/2}},    
		\end{equation}
		In order to use Theorem \ref{thm3}, we take
		\[
		(\eta_{1}, a_{1}) := (\gamma, l), \quad \quad (\eta_{2}, a_{2}) := (2, -(m+n-2)), \quad \text{and}  \quad (\eta_{3}, a_{3}) := (\sqrt{2}, -1).
		\]
		The number field containing $\eta_{1}, \eta_{2}, \eta_{3}$ is  $\mathbb{L} := \mathbb{Q}(\sqrt{2})$, which has degree $d_{\mathbb{L}} = [\mathbb{L}:\mathbb{Q}]=2$. Here 
		\begin{equation}\label{eq 3.33}
			\Lambda_{3} := \gamma^{l} 2^{-(m+n-2)} (\sqrt{2})^{-1} - 1,
		\end{equation}
		is nonzero. If  $\Lambda_{3} = 0$, then $\gamma^{2l} = 2^{2m+2n-3} \in \mathbb{Q}$ which is impossible. Therefore, $\Lambda_{3} \neq 0$. Moreover, since 
		\[
		h(\eta_{1}) = \frac{\log \gamma}{2}, \quad \quad h(\eta_{2}) = \log 2, \quad \quad  h(\eta_{3}) = \frac{\log 2}{2}.
		\]
		Therefore, we may take
		\[
		B_{1} := \log \gamma, \quad B_{2} := 2 \log 2 \quad \text{and} \quad B_{3} := \log 2.
		\]
		Finally, the fact that $m \leq n$ and the inequality \eqref{eq 3.13} imply that we can take $D:= 2 n$. Thus, taking into account inequality \eqref{eq 3.32} and applying Theorem \ref{thm3}, we obtain  
		\begin{equation}\label{eq 3.34}
			k < 1.1 \times 10^{13} \log n.    
		\end{equation}
		where we have used the fact that $1 + \log 2n < 2.1 \log n$ for $n \geq 5$. On the other hand, from  Lemma \ref{lem 3.1} we get 
		\[
		\log n < \log (5.1 \times 10^{31} k^{8} \log^{5} k ) < 73.01 + 8 \log k + 5 \log (\log k) < 75 \log k
		\]
		for $k \geq 3$. So, from \eqref{eq 3.34} we obtain
		\[
		k < 8.26 \times 10^{14} \log k.
		\]
		Further using Lemma \ref{lem 2.6} and Lemma \ref{lem 3.1} the above inequality leads to 
		\[
		k < 5.7 \times 10^{16} \quad \text{and} \quad n < 4.86 \times 10^{173}.
		\]
		This completes the proof Lemma \ref{lem 3.3}.
	\end{proof}
	The bounds obtained for $n$ and $k$ from the above lemma are very large. Now, our next goal is to reduce this upper bound to a reasonable range for which we will prove the following lemma.
	
	\begin{lemma}\label{lem 3.4}
		The equation \eqref{eq 1.4} has no solutions for $k > 450$ and $n \geq k+2$.
	\end{lemma}
	\begin{proof}
		Here, we attempt to reduce the upper bound on $n$ and $k$. To do so, we first consider
		\[
		\Gamma_{3} :=  -(m+n-2) \log 2 + l \log \gamma - \log \sqrt{2}.
		\]
		Since $k > 450$, then from \eqref{eq 3.32} we have $|\Lambda_{3}| < 0.01$. Choosing $a: = 0.01$, we obtain the inequality
		\begin{equation}\label{eq 3.35}
			|\Gamma_{3}| = |\log (\Lambda_{3} + 1)| < \frac{- \log(1 - 0.01)}{0.01} \cdot \frac{4}{2^{k/2}} < \frac{4.1}{2^{k/2}}.   
		\end{equation}
		by Lemma \ref{lem 2.5}. Thus, it follows that
		\begin{equation*}
			\left| -(m+n-2) \log 2 + l \log \gamma -  \log \sqrt{2}  \right| <  \frac{4.1}{2^{k/2}}.
		\end{equation*}
		Dividing the above inequality by $\log 2$, we get 
		\begin{equation}\label{eq 3.36}
			\left|l \left( \frac{\log \gamma}{\log 2} \right) - (m+n-2) -  \frac{1}{2} \right| < \frac{5.92}{2^{k/2}}. 
		\end{equation}
		Applying Lemma \ref{lem 2.2} with parameters
		\[ \tau :=  \frac{\log \gamma}{\log 2},  \quad \mu := -\frac{1}{2} , \quad A:= 5.92 \quad \text{and} \quad B:= 2.
		\]
		The fact $m+n-2 < 2n$ and $n := 4.86 \times 10^{173}$ by Lemma \ref{lem 2.2} imply that we can take $M := 9.72 \times 10^{173}$. We found that $q_{327}$, the denominator of the $327th$ convergent of $\tau$ exceeds $6M$. Furthermore, a quick computation with \textit{Mathematica} gives us the value
		\[
		\frac{\log(Aq_{327}/\epsilon)}{\log B}
		\]
		is less than $590$. So, if the inequality \eqref{eq 3.36} has a solution, then
		\[
		\frac{k}{2} < \frac{\log(Aq_{327}/\epsilon)}{\log B} < 590,
		\]
		which implies that $k \leq 1180$. With the above upper bound for $k$ and by Lemma \ref{lem 3.1}, we have
		\[
		n < 3.4 \times 10^{60}.
		\]
		We now proceed as we did before, but with $M := 6.8 \times 10^{60}$, we obtain that $k < 420$, which contradicts our assumption $k > 450$. This completes the proof of Lemma \ref{lem 3.4}.   
	\end{proof}
	
	\subsection{The final step}
	According to Lemma \ref{lem 3.2} and \ref{lem 3.4}, if $(l, k, n, m)$ is a solution of the Diophantine equation \eqref{eq 1.4} then
	\[
	3 \leq k \leq 450, \quad k+2 \leq n \leq 409, \quad 1 \leq m \leq n \quad \text{and} \quad 1 \leq l \leq 327.
	\]
	Using \textit{Mathematica}, we checked that all the solutions of the Diophantine equation \eqref{eq 1.4} are those listed in the statement of Theorem \ref{thm1}. This completes the proof of Theorem \ref{thm1}.
	
	\section{Proof of Theorem \ref{thm2}}
	We begin our analysis of \eqref{eq 1.5} for $2 \leq n \leq k+1$. In this case, it is known that $F_{n}^{(k)} = 2^{n-2}$, thus the equation \eqref{eq 1.5} transform into
	\[
	C_{l} = 2^{m+n-4},
	\]
	which has no solution in the range $2 \leq n \leq k+1$. From now on, we consider that $n \geq k+2$ and $k \geq 2$.
	\subsection{An upper bound for l versus n}
	The inequalities \eqref{eq 1.2} and \eqref{eq 2.10} together with the equation \eqref{eq 1.5}, imply
	\[
	\frac{\gamma^{l}}{2} \leq C_{l} =  F_{n}^{(k)} F_{m}^{(k)} \leq \varphi^{n+m-2} < \varphi^{2n-2},
	\]
	so we deduce that
	\[
	l \leq \left( \frac{\log 2}{\log \gamma}  \right) + (2 n-2) \left( \frac{\log \varphi}{\log \gamma}  \right).
	\]
	Using the fact that $2(1-2^{-k}) < \varphi(k) < 2$ for all  $k \geq 2$, we obtain
	\begin{equation}\label{eq 4.37}
		l < 0.8 n -0.4.    
	\end{equation}
	
	\subsection{Bounding $n$ in terms of $k$}
	In this step, we will give an inequality for $n$ in terms of $k$. More precisely, we will show the following lemma.
	\begin{lemma}\label{lem 4.1}
		If $(l, k, n, m)$ is an integer solution of \eqref{eq 1.5} with $k \geq 2$ and $n \geq k+2$, then the inequalities 
		\begin{equation}\label{eq 4.38}
			n < 3.28 \times 10^{32} k^{8} \log^{5} k
		\end{equation}
		hold.
	\end{lemma}
	\begin{proof}
		We use identities \eqref{eq 1.1} and \eqref{eq 2.9} to express \eqref{eq 1.5} into the form
		\begin{equation}\label{eq 4.39}
			\frac{\gamma^{l} + \delta^{l}}{2} = 
			\left( f_{k}(\varphi) \varphi^{n-1} +e_{k}(n) \right) \left(f_{k}(\varphi) \varphi^{m-1} +e_{k}(m) \right),
		\end{equation}
		which we rewrite as
		\[
		\frac{\gamma^{l}}{2} - f_{k}^{2}(\varphi) \varphi^{n+m-2}  = f_{k}(\varphi) \varphi^{n-1} e_{k}(m) + f_{k}(\varphi) \varphi^{m-1} e_{k}(n) + e_{k}(n) e_{k}(m) + \frac{\delta^{l}}{2}. 
		\]
		Therefore, we get
		\begin{equation}\label{eq 4.40}
			\left| \frac{\gamma^{l}}{2} - f_{k}^{2}(\varphi) \varphi^{n+m-2} \right| < \frac{f_{k}(\varphi)}{2} \varphi^{n-1} + \frac{f_{k}(\varphi)}{2} \varphi^{m-1} + \frac{1}{4} + \frac{|\delta|^{l}}{2}.
		\end{equation}
		Dividing both sides by $f_{k}^{2}(\varphi) \varphi^{n+m-2}$ and using the fact $f_{k}(\varphi) > 1/2$, we obtain
		\begin{equation}\label{eq 4.41}
			|\Lambda_{4}| <  \frac{1}{\varphi^{m-1}} + \frac{1}{\varphi^{n-1}} + \frac{3}{\varphi^{n+m-2}} < \frac{5}{\varphi^{m-1}},
		\end{equation}
		where
		\begin{equation}\label{eq 4.42}
			\Lambda_{4} :=  \gamma^{l} \varphi^{-(n+m-2)} \frac{ f_{k}^{-2}(\varphi)}{2} -1.
		\end{equation}
In a similar manner used to show that $\Lambda_{1} \neq 0$, one can verify that $\Lambda_{4} \neq 0$. To apply Theorem \ref{thm3}, we set
		\[
		\eta_{1} := \gamma, \quad \eta_{2} := \varphi , \quad \eta_{3} := \frac{f_{k}^{-2}(\varphi)}{2},
		\]
		and
		\[ 
		a_{1}:= l, \quad a_{2}:= -(n+m-2), \quad a_{3}:= 1.
		\]
		The number field containing $\eta_{1}, \eta_{2}, \eta_{3}$ is  $\mathbb{L} := \mathbb{Q}(\varphi, \sqrt{2})$, which has degree $d_{\mathbb{L}} = [\mathbb{L}:\mathbb{Q}]=2k$. As before, we have
		\[
		h(\eta_{1}) = \frac{\log \gamma}{2} \quad \text{and} \quad h(\eta_{2}) = \frac{\log \varphi}{k} < \frac{\log 2}{k}.
		\]
		Moreover
		\[
		\max\{2kh(\eta_{1}),|\log \eta_{1}|,0.16\} = k \log \gamma := B_{1}
		\]
		and 
		\[
		\max\{2kh(\eta_{2}),|\log \eta_{2}|,0.16\} = 2 \log 2 := B_{2}.
		\]
		Using the properties of logarithmic height and the estimate \eqref{eq 2.11}, we obtain
		\begin{align*}
			h(\eta_{3})  &  \leq 2 h(f_{k}(\varphi)) + h(2) \\
			& < 2 \left( \log(k+1) + \log 4 \right) +  \log 2 \\
			& < 8.2 \log k  
		\end{align*}
		for $k \geq 2$. So, we can take
		\[
		\max\{2kh(\eta_{3}),|\log \eta_{3}|,0.16\} = 16.4 k \log k := B_{3}.
		\]
		Finally, the fact that $m \leq n$ and the inequality \eqref{eq 4.37} implies that we can take $D:= 2 n$. Therefore, due to Theorem \ref{thm3} one obtains 
		\begin{equation}\label{eq 4.43}
			\log |\Lambda_{4}| >   -1.432 \times 10^{11} (2k)^{2} (1+ \log 2k) (1 +\log 2n) (k \log \gamma) (2 \log 2) (16.4 k \log k).   
		\end{equation}
		Comparing the lower bound \eqref{eq 4.43} and upper bound \eqref{eq 4.41} of   $|\Lambda_{4}|$ entails us
		\begin{equation}\label{eq 4.44}
			(m-1)  \log \varphi < 1.86 \times 10^{14} k^{4} \log^{2} k \log n,    
		\end{equation}
		where we have used the facts $1 + \log 2k < 3.5 \log k$ for all $k \geq 2$ and $1+ \log 2n < 2.3 \log n$ for all $n \geq 4$. 
		
		To apply Theorem \ref{thm3} the second time, we go back to \eqref{eq 1.5} and we rewrite it as
		\[
		\frac{\gamma^{l} + \delta^{l}}{2} = 
		\left( f_{k}(\varphi) \varphi^{n-1} +e_{k}(n) \right) F_{m}^{(k)},
		\]
		which yields
		\begin{equation}\label{eq 4.45}
			\left| \frac{\gamma^{l}}{2 F_{m}^{(k)}} - f_{k}(\varphi) \varphi^{n-1} \right|= \left| \frac{\delta^{l}}{2 F_{m}^{(k)}} + e_{k}(n) \right|  \leq 1.
		\end{equation}
		If we divide through by $f_{k}(\varphi) \varphi^{n-1}$, we obtain 
		\begin{equation}\label{eq 4.46}
			|\Lambda_{5}| \leq  \frac{1}{f_{k}(\varphi) \varphi^{n-1}}  < \frac{2}{\varphi^{n-1}}, 
		\end{equation}
		where
		\begin{equation}\label{eq 4.47}
			\Lambda_{5} := \gamma^{l}  \varphi^{-(n-1)} \frac{ f^{-1}_{k}(\varphi)}{2 F_{m}^{(k)} } - 1.
		\end{equation}
		We have $\Lambda_{5} \neq 0$.  Now, we apply Theorem \ref{thm3} to $\Lambda_{5}$ given by \eqref{eq 4.47} by fixing
		\[
		(\eta_{1}, a_{1}) := \left(\gamma, l \right), \quad \quad (\eta_{2}, a_{2}) := \left( \varphi, -(n-1) \right),~ \quad {\rm and} ~  \quad (\eta_{3}, a_{3}) := \left(\frac{f^{-1}_{k}(\varphi)}{2 F_{m}^{(k)} }, 1 \right).
		\]
		It is obvious that $\mathbb{L} := \mathbb{Q}(\varphi, \sqrt{2})$ contains $\eta_{1}, \eta_{2}, \eta_{3}$ and $d_{\mathbb{L}} = 2k$. As calculated before, we can take $B_{1} = k \log \gamma$ and  $B_{2} = 2 \log 2$. We need to compute $B_{3}$. Using the estimates \eqref{eq 2.10}, \eqref{eq 4.44} and properties of logarithmic height, we have for all $k \geq 2$,
		\begin{align*}
			h(\eta_{3})  & \leq h\left( \frac{ f^{-1}_{k}(\varphi)}{2 F_{m}^{(k)}} \right) \\
			& \leq  h(f_{k}(\varphi)) + h(2) + h(F_{m}^{(k)}) \\
			& <  \log(k+1) + \log4 + 2 \log 2 + (m-1) \log \varphi \\
			& < 1.87 \times 10^{14} k^{4} \log^{2} k \log n.  
		\end{align*}
		Thus, we deduce that
		\[
		\max\{2kh(\eta_{3}),|\log \eta_{3}|,0.16\} = 3.74 \times 10^{14} k^{5} \log^{2} k \log n := B_{3}.
		\]
		The inequality \eqref{eq 4.37} and the fact  $l < 0.8 n -0.4 < n$, which hold for all $n \geq 4$, allow us to select $D: = n$. Therefore, applying Theorem \ref{thm3} and comparing the resulting inequality with \eqref{eq 4.46}, we obtain
		\begin{equation}\label{eq 4.48}
			\frac{n}{\log^{2} n} < 8.18 \times 10^{27} k^{8} \log^{3} k,
		\end{equation}
		where we have used that $1 + \log 2k < 3.5 \log k$ for all $k \geq 2$ and $1+ \log n < 1.8 \log n$  for all $n \geq 4$. Thus, putting $S :=  8.18 \times 10^{27} k^{8} \log^{3} k$ in \eqref{eq 4.48} and using Lemma \ref{lem 2.6} together with the fact $64.27 + 8 \log k + 3 \log (\log k) < 100 \log k$ which holds for all $k \geq 2$, implies 
		\begin{align*}
			n & < 2^2 \left( 8.18 \times 10^{27} k^{8} \log^{3} k \right) (\log \left(8.18 \times 10^{27} k^{8} \log^{3} k \right))^2 \\
			& < 2^2 \left( 8.18 \times 10^{27} k^{8} \log^{3} k \right) (64.27 + 8 \log k + 3 \log (\log k))^2 \\
			&< 3.28 \times 10^{32} k^{8} \log^{5} k.
		\end{align*}
		This establishes \eqref{eq 4.38} and completes the proof of Lemma \ref{lem 4.1}.
	\end{proof}

	\subsection{The case when $2 \leq k \leq 500$}
	In this subsection, we will show the following result. 
	\begin{lemma}\label{lem 4.2}
		If $(l, k, n, m)$ is an integer solution of \eqref{eq 1.5} with $2 \leq k \leq 500$ and $n \geq k+2$, then our variables are bounded as follows:
		\[
		n  \leq 568 \quad \text{and} \quad l \leq 454.
		\]
	\end{lemma}
	\begin{proof}
		Let us define
		\begin{equation}\label{eq 4.49}
			\Gamma_{4} := \log (\Lambda_{4} + 1) = l \log \gamma -(n+m-2) \log \varphi + \log \left( \frac{ f_{k}^{-2}(\varphi)}{2} \right).   
		\end{equation}
		Assuming $m \geq 5$, from \eqref{eq 4.41}, we have $|\Lambda_{4}| < 0.98$. Choosing $a := 0.98$, we obtain the inequality
		\[
		|\Gamma_{4}| < \frac{- \log(1 - 0.98)}{0.98} \cdot \frac{5}{\varphi^{m-1}} < \frac{19.96}{\varphi^{m-1}}
		\]
		by Lemma \ref{lem 2.5}. Thus, it follows that
		\begin{equation*}
			\left|l \log \gamma -(n+m-2) \log \varphi + \log \left( \frac{f_{k}^{-2}(\varphi)}{2} \right) \right| <  \frac{19.96}{\varphi^{m-1}}. 
		\end{equation*}
		Dividing the above inequality by $\log \varphi$, we get 
		\begin{equation}\label{eq 4.50}
			\left| l  \left( \frac{\log \gamma}{\log \varphi}\right)  - (n+m) + 2 +  \frac{\log \left( \frac{f_{k}^{-2}(\varphi)}{2}\right)}{\log \varphi} \right| < \frac{49.23}{\varphi^{m-1}}.   
		\end{equation}
		We apply Lemma \ref{lem 2.2} with the parameters
		\[ \tau = \left( \frac{\log \gamma}{\log \varphi}\right), \quad \mu = 2 +  \frac{\log \left( \frac{f_{k}^{-2}(\varphi)}{2}\right)}{\log \varphi}, \quad A:= 49.23 \quad \text{and} \quad B:= \varphi.
		\]
		As seen before $\tau \notin \mathbb{Q}$. We take $ M_{k}:= \lfloor 3.28 \times 10^{32} k^{8} \log^{5} k    \rfloor$, which is an upper bound of $l$ from Lemma \ref{lem 4.1}. For each $k \in [2, 500]$, we find a good approximation of $\tau$ and a convergent $p_{l}/q_{l}$ of the continued fraction of $\tau$ such that $ q_{l} = q(k) > 6M_{k}$ is a denominator of a convergent of the continued fraction of $\tau$ with $ \epsilon = \epsilon(k):= \| \mu q\| - M_{k}\|\tau q\| > 0$. A computer search with \textit{Mathematica} revealed that if $k \in [2,500]$, then the maximum value of $\log(Aq/\epsilon)/ \log B $ is $553.311$, which is according to Lemma \ref{lem 2.2}, is an upper bound on $m-1$.  
		
		Now, we fix $1 \leq m \leq 554$ and we consider
		\begin{equation}\label{eq 4.51}
			\Gamma_{5} := \log (\Lambda_{5} + 1) = l \log \gamma -(n-1) \log \varphi + \log \left( \frac{ f_{k}^{-1}(\varphi)}{2 F_{m}^{(k)}}  \right).   
		\end{equation}
		Suppose that $n \geq 3$. Then from \eqref{eq 4.45}, we have $|\Lambda_{5}| < 0.89$. Choosing $a := 0.89$ and using Lemma \ref{lem 2.5}, we obtain
		\begin{equation}\label{eq 4.52}
			|\Gamma_{5}| < \frac{- \log(1 - 0.89)}{0.89} \cdot \frac{2}{\varphi^{n-1}} < \frac{4.97}{\varphi^{n-1}}.    
		\end{equation}
		Replacing \eqref{eq 4.52} by \eqref{eq 4.51} and dividing across by $\log \varphi$, we obtain
		\begin{equation}\label{eq 4.53}
			\left| l  \left( \frac{\log \gamma}{\log \varphi}\right)  - (n-1) +   \frac{\log \left( \frac{f_{k}^{-1}(\varphi)}{2 F_{m}^{(k)}}\right)}{\log \varphi} \right| <  \frac{12.26}{\varphi^{n-1}}. 
		\end{equation}
		To apply Lemma \ref{lem 2.2} to \eqref{eq 4.53}, this time for $1 \leq m \leq 554$, we set
		\[ \tau =  \frac{\log \gamma}{\log \varphi}, \quad \mu = 1 + \frac{\log \left( \frac{f_{k}^{-1}(\varphi)}{2 F_{m}^{(k)}}\right)}{\log \varphi}, \quad A:= 12.26 \quad \text{and} \quad B:= \varphi.
		\]
		Again, for $(k, m) \in [ 2, 500] \times [ 1, 554]$, we find a good
		approximation of $\tau$ and a convergent $p_{l}/q_{l}$ of the continued fraction of $\tau$ such that $ \epsilon = \epsilon(k):= \| \mu q\| - M_{k}\|\tau q\| > 0$, where $M_{k}:= \lfloor 3.28 \times 10^{32} k^{8} \log^{5} k \rfloor$ which is an upper bound of a from Lemma \ref{lem 4.1}. After doing this, we use Lemma \ref{lem 2.2}
		on inequality \eqref{eq 4.53}. Again, a program in Mathematica revealed that the maximum value of $\log(Aq/\epsilon)/ \log B$ over all $(k, m) \in [ 2, 500 ] \times [1, 554]$ is $567.728$, which according to Lemma \ref{lem 2.2}, is an upper bound of $n-1$.
		
		Hence, we deduce that the possible solutions $(l, k, n, m)$ of \eqref{eq 1.5} for which
		$k \in [2, 500]$ satisfy $k+2 \leq m \leq n \leq 568$. Therefore, we use inequality \eqref{eq 4.37} to obtain $l \leq 454$.
	\end{proof}
	
	\subsection{The case when $ k > 500$}
	This subsection aims to establish that there are no solutions for the case where $k > 500$ and $n \geq k+2$.
	
	\begin{lemma}\label{lem 4.3}
		If $(l, k, m, n)$ is a solution of the Diophantine equation \eqref{eq 1.5} with $k > 500$ and $n \geq k+2$, then $k$ and $n$ are bounded as
		\[
		k < 5.82 \times 10^{14} \quad \text{and} \quad n < 1.97 \times 10^{158}.
		\]
	\end{lemma}
	\begin{proof}
		For $k > 500$, it is easy to see that
		\[
		n < 3.28 \times 10^{32} k^{8} \log^{5} k < 2^{k/2}.
		\]
		Inserting \eqref{eq 1.1} and \eqref{eq 2.12} in \eqref{eq 1.5}, we obtain
		\begin{align*}
			\frac{\gamma^{l}}{2} -  2^{m+n-4} = 2^{m+n-4} \left( 2 \xi + \xi^{2} \right) -\frac{\delta^{l}}{2}.
		\end{align*}
		As $n \geq m \geq k+2$, this and the fact $\left( 2 \xi + \xi^{2} \right) < 3/2^{k/2}$ yield
		\begin{equation}\label{eq 4.54}
			|\Lambda_{6}| = \left| \gamma^{l} 2^{-(m+n-3)} - 1  \right| < \frac{3}{2^{k/2}} + \frac{1}{2^{2k}} < \frac{4}{2^{k/2}}.    
		\end{equation}
		But $\Lambda_{6}$ is not zero. Indeed, if $\Lambda_{6}$ were zero, we would then get that $\gamma^{l} = 2^{m+n-3} \in \mathbb{Q}$, which is impossible. Therefore, we can apply Theorem \ref{thm3} with
		\[
		(\eta_{1}, a_{1}) := (\gamma, l), \quad \text{and} \quad (\eta_{2}, a_{2}) := (2, -(m+n-3)).
		\] 
		Since $\eta_{1}, \eta_{2}$ are elements of the field $\mathbb{L} := \mathbb{Q}(\sqrt{2})$, then $d_{\mathbb{L}} = [\mathbb{L}:\mathbb{Q}]=2$. The fact that $m \leq n$ and the inequality \eqref{eq 4.37} implies that we can choose $D:= 2 n$ because $m+n-3 \leq 2n$. On the other hand, since $h(\eta_{1}) = \frac{\log \gamma}{2}$ and $h(\eta_{2}) = \log 2$. Thus, we can take
		$B_{1} := \log \gamma$ and $B_{2} := 2 \log 2.$ Thus, considering inequality \eqref{eq 4.54} and using Theorem \ref{thm3}, we get
		\begin{equation}\label{eq 4.55}
			k < 8.52 \times 10^{10} \log n,    
		\end{equation}
		where we have used the fact that $1 + \log 2n < 2.3 \log n$ for $n \geq 4$. In contrast, Lemma \ref{lem 4.1} leads to
		\[
		\log n < \log (3.28 \times 10^{32} k^{8} \log^{5} k ) < 74.87 + 8 \log k + 5 \log (\log k) < 114 \log k
		\]
		for $k \geq 2$. Thus, from \eqref{eq 4.54}, we get
		\[
		k < 9.72 \times 10^{12} \log k.
		\]
		Further using Lemma \ref{lem 2.6} and Lemma \ref{lem 4.1} the above inequality leads to 
		\[
		k < 5.82 \times 10^{14} \quad \text{and} \quad n < 1.97 \times 10^{158}.
		\]
		This finishes the proof Lemma \ref{lem 4.3}.
	\end{proof}
	
	The above lemma provides very large bounds for $n$ and $k$. Our next goal is to minimize this upper bound to a suitable range, for which we shall prove the following lemma.
	
	\begin{lemma}\label{lem 4.4}
		The equation \eqref{eq 1.5} has no solutions for $k \geq 500$ and $n \geq k+2$.
	\end{lemma}
	\begin{proof}
		Now, we lower the bound on $n$ and $k$. We define
		\[
		\Gamma_{6} := l \log \gamma -(m+n-3) \log 2.
		\] 
		Since $k > 500$, then from \eqref{eq 4.54} we have $|\Lambda_{6}| < 0.01$. Therefore by Lemma \ref{lem 2.5}, we have
		\begin{equation}\label{eq 4.56}
			|\Gamma_{6}| = |\log (\Lambda_{6} + 1)| < \frac{- \log(1 - 0.01)}{0.01} \cdot \frac{4}{2^{k/2}} < \frac{4.1}{2^{k/2}}.   
		\end{equation}
		Thus, it follows that
		\begin{equation*}
			\left| (m+n-3) \log 2 - l \log \gamma \right| <  \frac{4.1}{2^{k/2}}.
		\end{equation*}
		Dividing the above inequality by $\log \gamma$, we get 
		\begin{equation}\label{eq 4.57}
			\left| (m+n-3) \left( \frac{\log 2}{\log \gamma} \right) - l \right| < \frac{2.33}{2^{k/2}}. 
		\end{equation}
		To obtain a lower bound for the left-hand side of \eqref{eq 4.57}, we will apply Lemma \ref{lem 2.3} with
		\[
		\tau :=  \frac{\log 2}{\log \gamma} \notin \mathbb{Q}, \quad x:= m+n-3, \quad \text{and} \quad y:= l.
		\]
		The fact $m+n-3 < 2n$ and $n := 1.97 \times 10^{158}$ by Lemma \ref{lem 4.3} imply that we can take $M := 3.94 \times 10^{158}$. Let 
		\[
		[a_{0}, a_{1}, a_{2},\dots] := [0, 2, 1, 1, 5, 3, 2, 1, 22, 1, 5, 38, 1, 1, 1, 8, 1, 3, 7, 1, 5, 2, \
		5, 2, 2, 200, 1, 4, 4, 6  \dots]
		\]
		be the continued fraction expansion of $\tau$. A quick search using \textit{Mathematica} reveals that 
		\[
		q_{301} < 3.94 \times 10^{158} < q_{302}.
		\]
		Furthermore, $a_{M}:= \max\{a_{i} : i=0, 1, \dots,302\} = 4008$. So by Lemma \ref{lem 2.3}, we have
		\begin{equation}\label{eq 3.38}
			\left|(m+n-3)\left(\frac{\log 2}{\log \gamma} \right) - l \right| > \frac{1}{4010(m+n-3)}.
		\end{equation}
		Comparing estimates \eqref{eq 3.38}, \eqref{eq 4.57}, and by Lemma \ref{lem 4.3}, we get
		\[
		\frac{k}{2} < \frac{\log \left(2.33 \times 4008 \times 3.94 \times 10^{158} \right)}{\log 2} < 541,
		\]
		which implies that $k \leq 1082$. With this new upper bound of $k$ and by Lemma \ref{lem 4.1}, we have
		\[
		n < 1.1 \times 10^{61}.
		\]
		We now proceed as we did before but with $M:= 2.2 \times 10^{61}$ and we obtain  $k < 430$, which contradicts our assumption $k > 500$. This completes the proof of Lemma \ref{lem 4.4}.    
	\end{proof}
	\subsection{The final step}
	According to Lemma \ref{lem 4.2} and \ref{lem 4.4}, if $(l, k, n, m)$ is a solution of the Diophantine equation \eqref{eq 1.5} then
	\[
	2 \leq k \leq 500, \quad k+2 \leq n \leq 568, \quad 1 \leq m \leq n \quad \text{and} \quad 1 \leq l \leq 454.
	\]
Using \textit{Mathematica}, we checked that all the solutions of the Diophantine equation \eqref{eq 1.5} are those listed in the statement of Theorem \ref{thm2}. This completes the proof of Theorem \ref{thm2}.

\section*{Acknowledgment}
The author, Bijan Kumar Patel, acknowledges the Mukhyamantri Research Innovation (MRI) under MRIP, OSHEC, Government of Odisha, for extramural research funding (Grant No. 24EM/MT/90, 2024) to support this research work.

\vspace{05mm} \noindent \footnotesize
\begin{minipage}[b]{90cm}
\large{Department of Mathematics,\\
		School of Applied Sciences, \\ 
		KIIT University, Bhubaneswar, \\ 
		Bhubaneswar 751024, Odisha, India. \\
		Email: bptbibhu@gmail.com}
\end{minipage}

\vspace{05mm} \noindent \footnotesize
\begin{minipage}[b]{90cm}
\large{P. G. Department of Mathematics,\\
		Government Women's College, Sundargarh,\\ 
		Sambalpur University, Odisha, India. \\
		Email: iiit.bijan@gmail.com}
\end{minipage}

\end{document}